\documentclass[12pt]{article}
\usepackage{amssymb, amsmath, url, graphicx, setspace, geometry, amsthm, comment}

\newtheorem{theorem}{Theorem}[section]
\newtheorem{lemma}[theorem]{Lemma}

\def\3{\subset }
\def\4{\subseteq }
\def\<{\left<}
\def\>{\right>}

\def\bit{\begin{itemize}}
\def\eit{\end{itemize}}
\def\3{\subset }
\def\4{\subseteq }

\def\0{\leqno}

\def\barr{\begin{array}}
\def\earr{\end{array}}

\def\Z{{\rlap{$\kern2pt{\rm Z}$}{\rm Z}\,}}

\title{\bf A remark on the number of automorphisms of finite abelian groups}
\author{Marius T\u arn\u auceanu}
\date{November 1, 2024}

\begin{document}

\maketitle

\begin{abstract}
Let $Ab_0$ be the class of finite abelian groups and consider the function $f:Ab_0\longrightarrow(0,\infty)$ given by $f(G)=\frac{|{\rm Aut}(G)|}{|G|}$\,, where ${\rm Aut}(G)$ is the automorphism group of a finite abelian group $G$. In this short note, we prove that the image of $f$ is a dense set in $[0,\infty)$.
\end{abstract}



\section{Introduction}

A well-known question in group theory (see e.g. Problem 12.77 of \cite{5}) asks whether it is true that $|G|$ divides $|{\rm Aut}(G)|$ for every nonabelian finite $p$-group $G$. This was answered in negative in \cite{3}, where for each prime $p$ there was constructed a family $(G_n)_{n\in\mathbb{N}}$ of finite $p$-groups such that $|{\rm Aut}(G_n)|/|G_n|$ tends to zero as $n$ tends to infinity. By considering the function
\begin{equation}
f(G)=\frac{|{\rm Aut}(G)|}{|G|}\nonumber
\end{equation}for all finite groups $G$, the above result means that zero is an accumulation point of the image of $f$. Note that a similar result holds for the function
\begin{equation}
f'(G)=\frac{|{\rm Aut}(G)|}{\varphi(|G|)}\,,\nonumber
\end{equation}where $\varphi$ denotes the Euler's totient function (see e.g. \cite{1,2}). These constitute the starting point for our work.

Our main result shows that all nonnegative real numbers are accumulation points of the image of $f$, even if we restrict this function only to abelian groups.

\begin{theorem}
The set 
\begin{equation}
{\rm Im}(f)=\{f(G)\mid G\in Ab_0\}\nonumber 
\end{equation}is dense in $[0,\infty)$.
\end{theorem}

The proof of Theorem 1.1 follows the same steps as the proofs of Theorem 1.1 in \cite{4}. It is based on the next lemma which is a
consequence of Proposition outlined on p. 863 of \cite{6}.

\begin{lemma}
Let $(x_n)_{n\geq 1}$ be a sequence of positive real numbers such that $\lim_{n\rightarrow\infty}x_n=0$ and $\sum_{n=1}^{\infty}x_n$ is divergent. Then the set containing the sums of all finite subsequences of $(x_n)_{n\geq 1}$ is dense in $[0,\infty)$.
\end{lemma}

\noindent It also uses the fact that the function $f$ is multiplicative, that is 
\begin{equation}
f(G_1\times G_2)=f(G_1)f(G_2),\nonumber 
\end{equation}for any finite groups $G_1$, $G_2$ of coprime orders.
\bigskip

Finally, we formulate a natural open problem related to the above theorem.

\bigskip\noindent{\bf Open problem.} Is it true that for every $a\in [0,\infty)\cap\mathbb{Q}$ there is a finite (abelian) group $G$ such that $f(G)=a$?

\section{Proofs of the main results}

First of all, we prove an auxiliary result.

\begin{lemma}
The set ${\rm Im}(f)\cap [0,1]$ is dense in $[0,1]$.
\end{lemma}

\begin{proof}
Let $I$ be a finite subset of $\mathbb{N}$ and $p_i$ be the $i$th prime number. Since $f$ is multiplicative, we have
\begin{equation}
f\left(\prod_{i\in I}C_{p_i}\right)=\prod_{i\in I}f(C_{p_i})=\prod_{i\in I}\frac{p_i-1}{p_i}\nonumber
\end{equation}and so 
\begin{equation}
A=\left\{\prod_{i\in I}\frac{p_i-1}{p_i}\,\bigg|\, I\subset\mathbb{N}, |I|<\infty, p_i=\text{$i$th prime number}\right\}\subseteq {\rm Im}(f)\cap [0,1].\nonumber
\end{equation}Thus it suffices to prove that $A$ is dense in $[0,1]$. 

Consider the sequence $(x_i)_{i\geq 1}\subset (0,\infty)$, where $x_i=\ln(\frac{p_i}{p_i-1})$ for all $i\geq 1$. Clearly, $\lim_{i\rightarrow\infty}x_i=0$. We have
\begin{equation}
\lim_{i\rightarrow\infty}\frac{x_i}{\frac{1}{p_i}}=1.\nonumber
\end{equation}Therefore, since the series $\sum_{i\geq 1}\frac{1}{p_i}$ is divergent, we deduce that the series
$\sum_{i\geq 1}x_i$ is also divergent. So, all hypotheses of Lemma 1.2 are satisfied, implying that
\begin{equation}
\overline{\left\{\sum_{i\in I}x_i \,\bigg|\, I\subset\mathbb{N}^*, |I|<\infty\right\}}=[0,\infty).\nonumber
\end{equation}This means
\begin{equation}
\overline{\left\{\ln\left(\prod_{i\in I}\frac{p_i}{p_i-1}\right) \,\bigg|\, I\subset\mathbb{N}^*, |I|<\infty, p_i=\text{$i$th prime number}\right\}}=[0,\infty)\nonumber
\end{equation}or equivalently
\begin{equation}
\overline{\left\{\prod_{i\in I}\frac{p_i}{p_i-1} \,\bigg|\, I\subset\mathbb{N}^*, |I|<\infty, p_i=\text{$i$th prime number}\right\}}=[1,\infty).\nonumber
\end{equation}Then
\begin{equation}
\overline{\left\{\prod_{i\in I}\frac{p_i-1}{p_i} \,\bigg|\, I\subset\mathbb{N}^*, |I|<\infty, p_i=\text{$i$th prime number}\right\}}=[0,1]\nonumber
\end{equation}and consequently
\begin{equation}
\overline{A}=[0,1],\nonumber
\end{equation}as desired.
\end{proof}

We remark that the conclusion of Lemma 2.1 remains valid if we restrict $f$ to the class $Ab_0'$ of finite abelian groups of odd order, that is ${\rm Im}(f|_{Ab_0'})\cap [0,1]$ is also dense in $[0,1]$. 
\bigskip

We are now able to prove our main result.

\bigskip\noindent{\bf Proof of Theorem 1.1.} We have to prove that for every $a\in [0,\infty)$ and every $\varepsilon>0$ there is $G\in Ab_0$ such that $f(G)\in (a-\varepsilon,a+\varepsilon)$.

If $a\in [0,1]$, this follows from Lemma 2.1. Assume now that $a\in (1,\infty)$. Since ${\rm Aut}(C_2^n)\cong {\rm GL}_n(2)$ has order $\prod_{k=0}^{n-1}(2^n-2^k)$, we have
\begin{equation}
\lim_{n\rightarrow\infty}f(C_2^n)=\lim_{n\rightarrow\infty}\frac{1}{2^n}\prod_{k=0}^{n-1}(2^n-2^k)=\infty\nonumber
\end{equation}and so we can choose a finite elementary abelian $2$-group $G_1$ such that $f(G_1)=b>a$. Then $\frac{a}{b}\in (0,1)$. Let $\varepsilon_1=\frac{\varepsilon}{b}$\,. By the above remark, there is a finite abelian group of odd order $G_2$ with $f(G_2)\in \left(\frac{a}{b}-\varepsilon_1,\frac{a}{b}+\varepsilon_1\right)$. It follows that $G=G_1\times G_2\in Ab_0$ and
\begin{equation}
f(G)=f(G_1)f(G_2)=bf(G_2)\in (a-\varepsilon,a+\varepsilon).\nonumber 
\end{equation}This completes the proof.\qed

\vspace*{3ex}\small

\hfill
\begin{minipage}[t]{5cm}
Marius T\u arn\u auceanu \\
Faculty of  Mathematics \\
``Al.I. Cuza'' University \\
Ia\c si, Romania \\
e-mail: {\tt tarnauc@uaic.ro}
\end{minipage}

\end{document}